\documentclass[a4paper, reqno]{amsart}

\title[]{The Kazhdan-Lusztig category of $\W$-algebras of simply-laced Lie algebras at irrational levels}

\author{Thomas Creutzig}
\address{Friedrich-Alexander-Universität Erlangen-Nürnberg}
\email{thomas.creutzig@fau.de}
\author{Gurbir Dhillon}
\address{University of California, Los Angeles}
\email{gsd@math.ucla.edu}

\author{Shigenori Nakatsuka}
\address{Friedrich-Alexander-Universität Erlangen-Nürnberg}
\email{shigenori.nakatsuka.2022@gmail.com}

\usepackage{latexsym}
\usepackage{amsmath}
\usepackage{amsfonts}
\usepackage{amssymb}
\usepackage{amsxtra}
\usepackage{mathrsfs}
\usepackage{eucal}
\usepackage[all]{xy}
\usepackage{color}
\usepackage{braket}
\usepackage{graphics, setspace}
\usepackage{etoolbox, textcomp}
\usepackage{comment}
\usepackage{changepage}
\usepackage{xcolor}
\definecolor{rouge}{rgb}{0.85,0.1,.4}
\definecolor{bleu}{rgb}{0.1,0.2,0.9}
\definecolor{violet}{rgb}{0.7,0,0.8}
\usepackage[colorlinks=true,linkcolor=bleu,urlcolor=violet,citecolor=rouge]{hyperref}
\usepackage{cleveref}
\usepackage{lscape}
\usepackage{caption}
\usepackage{subcaption}
\usepackage{enumitem}
\usepackage[colorinlistoftodos]{todonotes}

\usepackage{tikz}			%dynkin-diagrams
\usetikzlibrary{matrix}
\usetikzlibrary{automata, arrows.meta, positioning,cd}
\usepackage{dynkin-diagrams}
\usetikzlibrary{cd}			%commutative-diagrams

\tikzset{%
  symbol/.style={
    draw=none,
    every to/.append style={
      edge node={node [sloped, allow upside down, auto=false]{$#1$}}
    },
  },
}

\newtheorem{definition}{Definition}[section]
\newtheorem{proposition}[definition]{Proposition}
\newtheorem{theorem}[definition]{Theorem}
\newtheorem{corollary}[definition]{Corollary}

\newtheorem{remark}[definition]{Remark}

\newtheorem{ex}[definition]{Example}

\numberwithin{equation}{section}

\newcommand{\tmath}[1]{\texorpdfstring{#1}{ABC}}
\newcommand{\Z}{\mathbb{Z}}
\newcommand{\Q}{\mathbb{Q}}
\newcommand{\C}{\mathbb{C}}

\newcommand{\W}{\mathcal{W}}

\newcommand{\g}{\mathfrak{g}}
\newcommand{\h}{\mathfrak{h}}

\newcommand{\nil}{\mathfrak{n}}

\newcommand{\sll}{\mathfrak{sl}}

\newcommand{\Hom}{\operatorname{Hom}}

\newcommand{\Ind}{\operatorname{Ind}}

\newcommand{\cC}{\mathcal C}
\newcommand{\cA}{\mathcal A}
\newcommand{\cR}{\mathcal R}
\newcommand{\cF}{\mathcal F}
\newcommand{\cG}{\mathcal G}

\newcommand{\cM}{\mathcal M}

\newcommand{\btimes}{\boxtimes}
\newcommand{\one}{\mathbf 1}
\newcommand{\Id}{\mathrm{Id}}

\renewcommand{\mod}{\operatorname{-mod}}

\newcommand{\fg}{\mathfrak{g}}
\newcommand{\gk}{\widehat{\fg}_\kappa}

\newcommand{\sC}{\mathcal{C}}
\newcommand{\Com}{\mathrm{Com}}
\newcommand{\kk}{\kappa}
\newcommand{\bM}{\mathbb{M}}
\newcommand{\io}{\mathcal{Y}(\cdot,z)}
\newcommand{\pr}{\mathrm{pr}}
\newcommand{\vect}{\mathrm{Vect}}

\newcommand{\KL}{\mathrm{KL}}
\newcommand{\weyl}{\mathbb{V}}

\newcommand{\Mod}{\text{-}\mathrm{mod}}
\newcommand{\cdo}[2]{\mathcal{D}^{\mathrm{ch,#2}}_{#1}}

\newcommand\doi[2]{\href{http://dx.doi.org/#1}{#2}}

\begin{document}
\maketitle
\begin{abstract}
    Let $\g$ be a simple, simply-laced Lie algebra and $f \in \g$ nilpotent. 
    The Kazhdan-Lusztig category of the $\W$-algebra $\W^\kappa(\g,f)$ associated with $(\g,f)$ at level $\kk \in \C$ is obtained from the Kazhdan-Lusztig category of the affine vertex algebra $V^\kk(\g)$ via quantum Hamiltonian reduction associated with $f$. We show that this is a braided tensor equivalence for any $f$  and any irrational level $\kk\in \C \backslash \Q$. 
\end{abstract}

\section{Introduction}

Let $\g$ be a simple Lie algebra and $\kappa$ a complex number. Then to this data one associates the affine vertex algebra of $\g$ at level $\kappa$, $V^\kappa(\g)$. Modules for $V^\kappa(\g)$ are smooth modules of the affine Lie algebra of $\g$, $\hat\g$, at level $\kappa$. In particular any $\g$-module parabolically induces to a $V^\kappa(\g)$-module. The Kazhdan-Lusztig category at level $\kappa$, $\KL^\kappa(\g)$, is the category of smooth $\g$-modules at level $\kappa$ whose simple composition factors are the simple quotients of the parabolic inductions of finite-dimensional $\g$-modules. The famous result of Kazhdan-Lusztig is that this category for $ \kappa + h^\vee \not \in \mathbb Q_{\geq 0}$  is braided tensor equivalent to the category of finite-dimensional weight modules for
the quantum group $U_q(\g)$ of $\g$ at parameter $q = \text{exp}\left(\frac{\pi i }{r^\vee (\kappa + h^\vee)}\right)$ \cite{KL}. Here $h^\vee$ is the dual Coxeter number of $\g$ and $r^\vee$ its lacing number. 

Let $f$ be a nilpotent element in $\g$. Then via quantum Hamiltonian reduction $H_f$ one obtains the $\W$-algebra $\W^\kappa(\g,f)$ from $V^\kappa(\g)$ \cite{FF1, KW}. Moreover the reduction of a $V^\kappa(\g)$-module is automatically a  $\W^\kappa(\g,f)$-module. Let $\KL^\kappa_f(\g)$ be the full subcategory generated by the image of $\KL^\kappa(\g)$ under reduction. We show
\begin{theorem}[{Theorem \ref{thm:main}}]\label{intro thm A}
 Let $\g$ be of type $ADE$ and $\kappa \in \C \backslash \Q$. Then, the BRST reduction $H_f\colon \KL^\kk(\g) \xrightarrow{\simeq} \KL_f^\kappa(\g)$ is an equivalence of braided tensor categories. 
\end{theorem}

In the theory of finite $\W$-algebras $U(\g,f)$, which is defined from the enveloping algebras $U(\g)$ of $\g$ in a similar way as the $\W$-algebras $\W^\kk(\g,f)$, it was proven by Losev \cite{Lo} that the reduction functor $H_f$ between the categories of Harish-Chandra bimodules over $U(\g)$ and $U(\g,f)$ is a tensor functor.
Theorem \ref{intro thm A} can be regarded as an upgrade in the vertex-algebraic context since Frenkel-Zhu's functor $A(-)$ sends the Kazhdan-Lusztig categories $\KL^\kk(\g), \KL_f^\kappa(\g)$ to the categories of Harish-Chandra bimodules and the the functor $A(-)$ commutes with $H_f$, thanks to a result by Arakawa \cite{A5}.

On the other hand, Theorem \ref{intro thm A} is an analogue at irrational levels of the result for exceptional $\W$-algebras at admissible levels \cite{ACF, AE}.
It, in particular, implies that the category $\KL^\kk_f(\g)$ does not depend on the choice of the representatives $f$ in a nilpotent orbit and especially the good gradings; the category $\KL^\kk_f(\g)$ is braided tensor equivalent to the category of weight modules of $U_q(\g)$.

Parameterizing the quantum group by $q$ is not misleading as the bialgebra structure only depends on $q$. Its category of finite dimensional weight modules becomes a braided tensor category with braiding described by the $R$-matrix. Here one makes the choice and the choice is in fact a root of unity of order $N$, where $N$ is the level of the weight lattice $P$ of $\g$, that is the smallest positive integer $N$, such that the lattice $NP$ is integral. 
Thus shifting $\frac{1}{k+h^\vee}$ by an even integer (for $\g$ simply laced) doesnot affect the $q$-parameter, but it affects the $R$-matrix, that is the braiding. This should amount to a cocycle twist by an element in $H^3(P/Q, \mathbb C^\times)$. 
See the introduction of \cite{Mo} for a discussion of this. 
We see this realized by "twisting" the categories by $\text{Vect}_{P/Q}$ in the following sense.
\begin{corollary}[{Corollary \ref{cor:main}}]
For $\g$ of type $ADE$, $\kk \in \C \backslash \Q$, define $\kk_m \in  \C \backslash \Q$ via  $$\frac{1}{\kk + h^\vee} = \frac{1}{\kk_m + h^\vee} +m,\quad (m\in \Z).$$
Then, there exists a braided tensor equivalence 
$$\KL^{\kk}_f(\g) \simeq \left(\KL_{f'}^{\kk_m}(\g)\boxtimes \vect^m_{P/Q}\right)^{P/Q}$$
for arbitrary nilpotent elements $f$ and $f'$.
\end{corollary}
Here $\text{Vect}^m_{P/Q}$ is the category of $P/Q$-graded vector spaces with braiding determined by the Killing form rescaled by $m$. 
in the case of $f=0$ and $f'$  principal, this corollary was conjectured by Aganagic-Frenkel-Okounkov in \cite[Conjecture 6.4]{AFO} without the abelian cocycle modifications corresponding to the factor $\vect^m_{P/Q}$.

So far braided tensor categories of $\W$-algebras have only been known for the Virasoro algebra \cite{CJORY}, the $N=1$ super Virasoro algebra \cite{CMOY} as well as exceptional simple $\W$-algebras at admissible levels \cite{A5, AE, C2}. The proof of our Theorem is similar to ideas of \cite{C2}. 
The Kazhdan-Lusztig category of universal $\W$-algebras at rational levels is usually not semi-simple and in particular in the case that its level shifted by the dual Coxeter number is a positive rational number its abelian structure is unknown, except for the recent studies of the case $\mathfrak{sl}_2$ \cite{MY}.

The first issue in proving such theorems has been the mere existence of a braided tensor category structure. Thanks to a recent Theorem of Yi-Zhi Huang \cite{H1} this is now guaranteed for the category of $C_1$-cofinite modules of any vertex operator algebra.  
Equivalences of categories of vertex algebras can now be established if one has a suitable relation between the involved vertex algebras. For us these are the GKO-coset realizations of principal $\W$-algebras of types $ADE$ \cite{ACL}. These allow us to determine important fusion rules of modules of principal $\W$-algebras. The modules $T^\kappa_{\lambda, \mu}$ are twisted quantum Hamiltonian reductions of Weyl modules \cite{AF} and we crucially show the fusion rule $T^\kappa_{\lambda, \mu} \cong T^\kappa_{\lambda, 0} \boxtimes T^\kappa_{0, \mu}$, see Theorem \ref{thm:fusion}.
This result together with the Urod-version of the GKO-coset \cite{ACF} then allows us to prove our main Theorem. 

One would like to obtain a similar result for the other simple Lie algebras and there are many more GKO-type cosets (see the introduction of \cite{CKL2}). In particular there are three GKO-type cosets that involve the affine vertex algebras of types $B, C$ and also $\mathfrak{osp}_{1|2n}$ as well as their principal $\W$-(super)algebras \cite{CKL, CGL, CL2}.
We have two works in progress. One on general structural results about $\W$-superalgebras and one on algebras of chiral differential operators on supergroups. These two will give us more detailed information about the GKO-cosets of orthosymplectic types  of \cite{CKL, CGL, CL2} and hence we expect to be able to report on a similar Theorem for types $B, C$ and also $\mathfrak{osp}_{1|2n}$ in the future. 

\subsection*{Acknowledgements} We are very grateful to Robert McRae and Cris Negron for lifting the result of Yi-Zhi Huang on the vertex tensor category structure on the category of $C_1$-cofinite modules \cite{H1} to its ind-completion \cite{MN}. Without their theorem, we could not have obtained our results. 

\section{Affine vertex algebras and $\W$-algebras}

\subsection{Affine vertex algebras}
Let $\g$ be a finite-dimensional simple Lie algebra, $\g=\nil_+\oplus \h \oplus \nil_-$ the standard triangular decomposition.
Let $Q$ denote the root lattice, $P$ (resp.\ $\check{P}$) the (co-)weight lattice, $P_+\subset P$ (resp.\ $\check{P}_+\subset \check{P}$) the set of dominant integral (co-)weights.

Let $\widehat{\g}_\kappa=\g(\!(t)\!)\oplus \C \mathbf{1}$ be the affine Kac-Moody Lie algebra associated with $\g$ at level $\kappa(\in \C)$, whose commutator is defined by 
\begin{align*}
[uf(t),vg(t)]=[u,v]f(t)g(t)+\kappa(u,v)\underset{t=0}{\mathrm{Res}}g(t)df(t) \mathbf{1},\quad [\widehat{\g}_\kappa,\mathbf{1}]=0
\end{align*}
where $(\cdot, \cdot)$ is a normalized invariant bilinear form on $\g$.
For a $\g$-module $M$, let 
\begin{align*}
\weyl^\kappa(\overline{M})=U(\gk)\underset{U(\g[\![t]\!]\oplus \C \mathbf{1})}{\otimes}\overline{M}
\end{align*}
be the parabolically induced $\gk$-module where $\overline{M}$ is regarded as a $(\g[\![t]\!]\oplus \C\mathbf{1})$-module where $\g[\![t]\!]$ acts by the natural projection $\g[\![t]\!]\twoheadrightarrow \g$ and $\mathbf{1}$ acts by the identity.
The Weyl modules $\weyl_\lambda^\kk$ with $\lambda \in \h^*$ are, by definition, those induced from the simple $\g$-module $L(\lambda)$ of  highest weight $\lambda$, i.e., $\weyl_\lambda^\kk=\weyl^\kk(L(\lambda))$.  
In particular, $V^\kappa(\g)=\weyl^\kk_0$ has a structure of vertex algebra, called the universal affine vertex algebra of $\g$ at level $\kk$.
Since $V^\kk(\g)$-modules are equivalent to smooth $\gk$-modules, $\weyl^\kk(\overline{M})$ are naturally $V^\kk(\g)$-modules. 

When the level $\kk$ is not critical, i.e., $\kk\neq- h^\vee$, $V^\g$ is equipped with a conformal vector $\omega_\g$ by the Segal-Sugawara construction. Note that the Fourier modes $L_n$ of the corresponding field $L(z)=\sum_{n \in \Z}L_n z^{-n-2}$, define the action of the Virasoro algebra on $V^\kk(\g)$ with semisimple $L_0$-action and $L_0$-eigenvalues are called conformal weights.

The Kazhdan-Lusztig category $\KL^\kappa(\g)$ of $V^\kappa(\g)$-modules is, by definition, the full subcategory consisting of those modules such that the $\g[\![t]\!]$-action  integrates to the group action of $G[\![t]\!]$ and the conformal weights are bounded from below.
When the level $\kappa$ is irrational, i.e., $\kappa\in \C\backslash \Q$, $\KL^\kappa(\g)$ is semisimple, and the simple objects are given by the Weyl modules $\weyl^\kappa_\lambda$ with $\lambda\in P_+$.

\subsection{\tmath{$\W$}-algebras}\label{sec: W-algebras}
Let $f$ be a nilpotent element in $\g$ and a good $\frac{1}{2}\Z$-grading $\Gamma\colon \g=\oplus_{d}\g_d$. 
Let us introduce the Lie subalgebras $\g_{\geq d}=\oplus_{d'\geq d}\,\g_{d'} \supset \g_{>d}=\oplus_{d'> d}\,\g_{d'}$ for $d\geqslant 0$.
Then, one associates a neutral fermion vertex algebra on the induced representation 
$$F_\chi=U(\g_{>0}(\!(t)\!))\otimes_{U(\g_{>0}[\![t]\!]+\g_{\geqslant 1}(\!(t)\!))}\C_{\chi}$$
where $\C_{\chi}$ is the one-dimensional representation of $\g_{>0}[\![t]\!]+\g_{\geqslant 1}(\!(t)\!)$ such that $\g_{>0}[\![t]\!]$ acts trivially and $\g_{\geqslant 1}(\!(t)\!)$ acts by the character $\chi(xt^n)=\delta_{n,-1}(f,x)$.

The $\W$-algebra $\W^\kk(\g,f)$ is, by definition, the vertex algebra realized on the $0$-th cohomology of Feigin's semi-infinite complex
\begin{align}\label{eq: BRST complex}
    C^{\frac{\infty}{2}+\bullet}(\g_{>0}(\!(t)\!); V^\kk(\g)\otimes F_\chi)
\end{align}
with the diagonal $\g_{>0}(\!(t)\!)$-action on $V^\kk(\g)\otimes F_\chi$.
This complex is also called the BRST complex \cite{FF1, KW}, which we denote by $C_f^\bullet(M)$ with $M=V^\kk(\g)$. Hence, the $\W$-algebra is defined as
$$\W^\kk(\g,f)=H^0(C_f^\bullet(V^\kk(\g))).$$
When the level $\kk$ is not critical, $\W^\kk(\g,f)$ is equipped with a conformal vector $\omega_{\g,f}$ (e.g. \cite[Section 2]{KW}) by using the Segal-Sugawara conformal vector $\omega_\g$ in $V^\kk(\g)$.

By replacing $V^\kk(\g)$ with its representations $M$ in \eqref{eq: BRST complex}, one obtains $\W^\kk(\g,f)$-modules $H_f^0(M)$, which defines a functor of representation categories
\begin{align}\label{eq: BRST reduction}
    H_{f}^0 \colon V^\kk(\g)\mod \rightarrow \W^\kk(\g,f)\mod.
\end{align}
In particular, let us introduce the $\W^\kk(\g,f)$-modules
$$\weyl_{\lambda,f}^\kk:=H_{f}^0(\weyl^\kk_\lambda)$$
for $\lambda \in P_+$.
In this paper, we study the category $\KL^\kappa_f(\g)$ of $\W^\kk(\g,f)$-modules whose objects are the direct sums of $\weyl_{\lambda,f}^\kk$ with $\lambda \in P_+$.
We note that $H_{f}^{\neq0}(\weyl^\kk(\overline{M}))=0$ holds for $\g$-modules $M$ by \cite{KW}.

\subsection{Twisted reductions}
The d.g. vertex superalgebra $C_f(V^\kk(\g))$ admits a diagonal embedding of the Heisenberg vertex algebra associated with the Cartan subalgebra $\h$, see \cite[Section 7]{ACF} for the details.
Therefore, we have a twisted action of $C_f(V^\kk(\g))$ on $C_f(M)$ for $V^\kk(\g)$-modules $M$, parametrized by the integrality condition $\check{\mu}\in \check{P}$. 
This new action equips $C_f(M)$ with a new differential, and thus a cohomology which we denote by $H_{f,\check{\mu}}^\bullet(M)$. 
Thanks to the twisted $C_f(V^\kk(\g))$-action on $C_f(M)$, $H_{f,\check{\mu}}^\bullet(M)$ is a module over the $\W$-algebra $\W^\kk(\g,f)$ \cite{ACF, AF}.

Later, we will use the case when $f$ is a principal nilpotent element. 
In this case, we will denote the corresponding $\W$-algebras by $\W^\kk(\g)$ and set
$$T_{\lambda,\check{\mu}}^\kk:=H_{f,\check{\mu}}^0(\weyl_{\lambda}^\kk),\quad (\lambda \in P_+,\ \check{\mu}\in \check{P}_+).$$

The Feigin-Frenkel duality \cite{FF1} asserts an isomorphism of vertex algebras
$$\W^\kk(\g)\simeq \W^{\check{\kk}}(\check{\g})$$
where $\check{\g}$ is the Langlands dual of $\g$ and the levels satisfie the relation $r^\vee(\kk+h^\vee)(\check{\kk}+\check{h}^\vee)=1$ with $r^\vee$ (resp.\ $h^\vee$) the lacing number (resp.\ the dual Coxeter number) of $\g$.
This duality is generalized to their representations $T_{\lambda,\check{\mu}}^\kk$ as follows.
\begin{theorem}[{\cite[Theorem 2.3]{AF}}]\label{thm: simplicity of T-modules}
Let $f$ be principal and $\kappa \in \C \backslash \Q$. 
For $\lambda \in P_+$ and $\check{\mu} \in \check{P}_+$, the $\W^\kk(\g)$-modules $T_{\lambda, \check{\mu}}^{\kappa}$ are simple and, moreover, there exist isomorphisms 
$$T^\kk_{\lambda,\check{\mu}}\simeq T^{\check{\kk}}_{\check{\mu},\lambda}$$
under the Feigin-Frenkel duality.
\end{theorem}

By \cite[Theorem 4.3]{AF}, the simple $\W^\kk(\g)$-module $T_{\lambda, \check{\mu}}^{\kappa}$ is a highest weight representation of highest weight $\chi(\lambda-(\kk+h^\vee) \check{\mu})$ and the conformal weight, that is the $L_0$-eigenvalue of the highest weight vector, is given by 
\begin{align*}
  h^{\kk}_\Lambda = \frac{(\Lambda,\Lambda+2\rho)}{2(\kappa+h^\vee)} - (\Lambda,\check{\rho}).  
\end{align*}
by setting $\Lambda=\lambda-(\kk+h^\vee) \check{\mu}$ where $\rho$ (resp.\ $\check{\rho}$) denotes the Weyl (co-)vector of $\g$, see e.g. \cite[Theorem 3.6]{KW}.
Then it is straightforward to derive the following formula, which will be used later:
\begin{equation}\label{eq:confweights}
    h(T_{\lambda,\check{\mu}}^{\kappa}) - h(T_{\lambda, 0}^{\kappa}) - h(T_{0, \check{\mu}}^{\kappa})
 = -(\lambda,\check{\mu}).
\end{equation}

\subsection{Equivariant \tmath{$\W$}-algebras at irrational levels}\label{sec: CDOs}

Given a (simply-connected, connected) simple algebraic group $G$ whose Lie algebra is $\g$, we have a vertex algebra, denoted by $\cdo{G}{\kappa}$, called the algebra of chiral differential operators at level $\kappa$.
It is a simple vertex algebra \cite[Section 9]{AM} and realized as global sections of sheaves of vertex algebras over $G$ whose local sections are the vertex-algebraic analogue of the differential operators on the domains, while the levels $\kappa$ parametrize the gluing conditions \cite{AG, GMS}. 
As an analogue of the algebra of global differential operators, which contain the enveloping algebra of $\g$ realized as left and right invariant vector fields, there is a vertex algebra homomorphism 
$$V^\kappa(\g)\otimes V^{\kk_o}(\g)\rightarrow\cdo{G}{\kappa}$$
where the level $\kk_{o}$ is determined by the formula 
\begin{align}\label{eq: opposit level}
    (\kappa+h^\vee)+(\kk_o+h^\vee)=0.
\end{align}
When $\kappa \in \C\backslash \Q$, it is an embedding, giving rise to a decomposition 
\begin{align*}
    \cdo{G}{\kappa}\simeq \bigoplus_{\lambda \in P_+} \weyl_{\lambda}^\kappa \otimes \weyl_{\lambda}^{\kk_o}
\end{align*}
as a $V^\kappa(\g)\otimes V^{\kk_o}(\g)$-module, an analogue of the Peter-Weyl theorem. 
The equivariant $\W$-algebras \cite{A4} are, by definition, the vertex algebras obtained from $\cdo{G}{\kappa}$ by applying the BRST reductions:
\begin{align*}
    \cdo{G,f}{\kappa}=H^0_{f}(\cdo{G}{\kappa}).
\end{align*}

\begin{theorem}[{\cite[Section 6]{A4}}]\label{thm: equiv Walg}
The equivariant $\W$-algebras $\cdo{G,f}{\kappa}$ are simple vertex algebras. 
For $\kappa \in \C \backslash \Q$, it is a conformal extension of $\W^\kk(\g,f) \otimes  V^{\kk_o}(\g)$ and decomposes into
\begin{align*}
\cdo{G,f}{\kappa}\simeq \bigoplus_{\lambda \in P_+} \weyl_{\lambda,f}^\kappa \otimes \weyl_{\lambda}^{\kk_o}
\end{align*}
as a $\W^\kappa(\g,f) \otimes  V^{\kk_o}(\g)$-module.
\end{theorem}

\subsection{Cosets à la Goddard-Kent-Olive}

By the Frenkel-Kac construction, the simple affine vertex algebra $L_1(\g)$ at level one is isomorphic to the lattice vertex algebra $V_Q$ associated with the root lattice $Q$. Goddard, Kent, and Olive \cite{GKO} observed that when $\g=\sll_2$, the Virasoro vertex algebra, i.e., $\W^\kappa(\sll_2)$ is realized as the following coset subalgebra
\begin{align*}
    \W^{\kk^o}(\sll_2)\simeq \mathrm{Com}(V^{\kk}(\sll_2),V^{\kk-1}(\sll_2)\otimes L_1(\sll_2)),
\end{align*}
or equivalently, a conformal embedding 
\begin{align*}
    V^{\kk}(\sll_2) \otimes \W^{\kk^o}(\sll_2)\hookrightarrow V^{\kk-1}(\sll_2)\otimes L_1(\sll_2),
\end{align*}
for irrational levels $\kappa$ where the other level $\kk^o$ is determined by the relation $\frac{1}{\kk+2}+\frac{1}{\kk^o+2}=1$.
The following is a similar coset realization of the principal $\W$-algebras for simply-laced Lie algebras, which can be regarded as a shifted version of the equivariant $\W$-algebras in the previous section. It has been first proven in \cite{ACL} and then been reproven in very different fashions afterwards. 

\begin{theorem}[\cite{ACL, CN, CL1}] \label{thm:GKO1} For simple Lie algebras $\g$ of type $ADE$ and $\kk \in \C \backslash \Q$, set $\kk^o \in \C \backslash \Q$ by the relation
\begin{equation}\label{eq: shifting the level}
    \frac{1}{\kappa + h^\vee} + \frac{1}{\kk^o + h^\vee} = 1. 
\end{equation} 
Then, there is a conformal embedding $V^\kk(\g) \otimes \W^{\kk^o}(\g)\hookrightarrow V^{\kk-1}(\g)\otimes V_Q$. 
Moreover, for $\lambda,\mu \in P_+$, it induces decompositions 
\begin{align*}
        \weyl_\lambda^{\kk-1} \otimes  V_{Q+\mu} \simeq \bigoplus_{\substack{\nu \in P_+ \\ \nu \equiv \lambda + \mu}} \weyl_\nu^\kk \otimes T^{\kk^o}_{\nu, \lambda}
\end{align*}
as $V^\kappa(\g) \otimes \W^{\kk^o}(\g)$-modules where $\nu \equiv \lambda +\mu$ is equality modulo $Q$.
\end{theorem}
\noindent 

By \cite{ACF}, tensoring the simple affine vertex algebra $L_n(\g)$ at $n \in \Z_{\geq0}$ and applying the quantum Hamiltonian reduction commute:
$$H_f^\bullet(V^\kk(\g)\otimes L_n(\g))\simeq H_f^\bullet(V^\kk(\g))\otimes L_n(\g).$$
When $\g$ is of type $ADE$ and $n=1$, Theorem \ref{thm:GKO1} implies the conformal embedding
$$\W^{\kk}(\g,f)\otimes \W^{\kk^o}(\g)\hookrightarrow \W^{\kk-1}(\g,f)\otimes V_Q.$$
The branching rules of tensor product representations under this conformal embedding are described as follows.
\begin{corollary}[\cite{ACF}] \label{cor:urod} 
    Let $\g$ be of type $ADE$ and $\kappa \in \C \backslash \Q$ with $\kk^o$ as in \eqref{eq: shifting the level}.
\begin{enumerate}[leftmargin=*]
\item  For an arbitrary nilpotent element $f$ and $\lambda,\mu \in P_+$
\begin{equation*}
\weyl_{\lambda, f}^{\kk-1} \otimes  V_{Q+\mu} \simeq \bigoplus_{\substack{\mu' \in P_+ \\ \mu' = \lambda + \mu}} \weyl_{\mu', f}^\kappa \otimes T^{\kk^o}_{\mu', \lambda}
\end{equation*}
as modules over $\W^{\kk}(\g,f)\otimes \W^{\kk^o}(\g)$.
\item  For a principal nilpotent element $f$ and $\lambda, \mu, \nu \in P_+$
\begin{equation*}
T_{\lambda, \mu}^{\kappa-1} \otimes  V_{Q+\nu} \simeq \bigoplus_{\substack{\mu' \in P_+ \\ \mu' \equiv \lambda +\mu+ \nu}} T_{\mu', \mu}^\kk \otimes T^{\kk^o}_{\mu', \lambda}
\end{equation*}
as modules over $\W^{\kk}(\g)\otimes \W^{\kk^o}(\g)$.
\end{enumerate}
\end{corollary}
\begin{comment}
\begin{corollary} \label{cor:urod} For irrational $\kappa$
    \begin{enumerate}[leftmargin=*]
    \item  For $\g$ simply laced and any nilpotent element $f$ and $\nu \in P_+$
\begin{equation}
\begin{split}
\weyl_{\lambda, f}^{\kappa-1} \otimes  V_{Q+\omega} &\cong \bigoplus_{\substack{\mu \in P_+ \\ \mu = \lambda + \omega\ \text{mod}\ Q }} \weyl_{\mu, f}^\kappa \otimes T^\ell_{\mu, \lambda}   \\   
T_{\lambda, \nu}^{\kappa-1} \otimes  V_{Q+\omega} &\cong \bigoplus_{\substack{\mu \in P_+ \\ \mu = \lambda +\nu+ \omega\ \text{mod}\ Q }} T_{\mu, \nu}^\kappa \otimes T^\ell_{\mu, \lambda}
\end{split}
\end{equation}
 The level $\ell$ is related to $\kappa$ via
\begin{equation}
    \frac{1}{\kappa + h^\vee} + \frac{1}{\ell + h^\vee} = 1. 
\end{equation}
\item For $\g = \mathfrak{sp}_{2n}$ and any nilpotent element $f$ and $\rho \in P^\vee_+$, 
\begin{equation}
\begin{split}
 \weyl_{\lambda, f}^{\kappa-1} \otimes  F(4n) &\cong \bigoplus_{\mu \in P_+} \bigoplus_{\nu \in P^\vee_+ \cap \mathbb Z^n }  \weyl_{\mu, f}^\kappa \otimes T^{\ell_1}_{\mu, \nu} \otimes T^{\ell_2}_{\lambda, \nu}  \\
 T_{\lambda, \rho}^{\kappa-1} \otimes  F(4n) &\cong \bigoplus_{\mu \in P_+} \bigoplus_{\nu \in P^\vee_+ \cap \mathbb Z^n }  T_{\mu, \rho}^\kappa \otimes T^{\ell_1}_{\mu, \nu} \otimes T^{\ell_2}_{\lambda, \nu} 
\end{split}
\end{equation}
with 
\[
\frac{1}{\kappa + h^\vee} + \frac{1}{\ell_1 + h^\vee} = 2, \qquad
(\ell_1 + h^\vee) + (\ell_2 + h^\vee) = 1.
\]
\end{enumerate}
\end{corollary}
\end{comment}

\section{Braided tensor categories}
We recall the structures of a vertex and braided tensor category on the representation categories of simple conformal vertex superalgebras, developed by Huang-Lepowsky-Zhang \cite{HLZ0}-\cite{HLZ2}. The main reference here is \cite[\S 2.4]{CY}. 
\subsection{Intertwining Operators}
Throughout this section, $(V,\omega)$ denotes a simple conformal vertex algebra $(V,\omega)$ such that the field  $L(z)=\sum_{n\in \Z} L_n z^{-n-2}$ corresponding to the conformal vector $\omega$ gives a conformal grading 
\begin{align*}
    V=\bigoplus_{\Delta \in \frac{1}{r}\Z_{\geq0}}V_\Delta,\quad V_0=\C \one
\end{align*}
for some $r\in \Z_{>0}$. The $V$-modules $M$ are always assumed to be lower-bounded, that is, it admits a generalized conformal grading 
$$M=\bigoplus_{\Delta \in \C}M_\Delta,\quad \dim M_\Delta<\infty$$
which is lower-bounded in the sense that there exist some $\lambda_1,\cdots, \lambda_n\in \C$ satisfying 
$$\{\Delta\in \C\mid M_\Delta\neq0\}\subset \bigsqcup \{\lambda_i+\frac{1}{r}\Z_{\geq0}\}.$$

Given a $V$-module $M$, we denote by $Y(\cdot, z)\colon V\times M \rightarrow M(\!(z)\!)$ the structure map.
The tensor product $M_1\boxtimes M_2$ of the $V$-modules $ M_1$ and $ M_2$ is defined as the object representing the space of (logarithmic) intertwining operators given below.

For $V$-modules $M_1, M_2, M_3$, a (logarithmic) \emph{intertwining operator} of type ${M_3\choose M_1\,M_2}$ is a bilinear map
\begin{align}\label{log:map0} \nonumber
\mathcal{Y}(\cdot,z): &M_1\times M_2\to M_3\{z\}[\log z] \\ \nonumber
&(u,v)\mapsto \sum_{k=0}^{K}\sum_{n\in
{\mathbb C}}{u}_{n, k}^{
\mathcal{Y}}v\ z^{-n-1}(\log z)^{k}
\end{align}
satisfying the following conditions for each 
$$a \in V,\quad u \in M_1, \quad v \in M_2, \quad n\in \C, \quad k=0,\cdots, K.$$
\begin{enumerate}[leftmargin=*]
\item The {\em lower truncation
condition}:
\begin{equation} \nonumber
{u}_{n+m, k}^{\mathcal{Y}}v=0\;\;\mbox{ for }\;m\in {\mathbb
N} \;\mbox{sufficiently large}.
\end{equation}
\item The {\em Jacobi identity}:
\begin{equation} \nonumber
\begin{split}
 z^{-1}_0\delta \bigg( {z_1-z_2\over x_0}\bigg)
Y(a,z_1){\mathcal{Y}(u,z_2)v}
- z^{-1}_0\delta \bigg( {z_2-z_1\over -z_0}\bigg)
{\mathcal{Y}}(u,z_2)Y(a,z_1)v \\
 = z^{-1}_2\delta \bigg( {z_1-z_0\over z_2}\bigg){
\mathcal{Y}}(Y(a,z_0)u,z_2) v.
\end{split}
\end{equation}
\item The {\em $L_{-1}$-derivative property:} 
\begin{equation}\nonumber
{\mathcal{Y}}(L_{-1}u,z)=\frac d{dz}{\mathcal{Y}}(u,z).
\end{equation}
\end{enumerate}
In the above, $\delta(z)=\sum_{n\in \Z} z^{n}$ denotes the $\delta$-function.
The intertwining operators satisfy the following property.
\begin{proposition}[{\cite[Proposition 2.3]{C1}}]\label{prop:intertwiner}
Let $M_1$, $M_2$ be $V$-modules  the subspaces generated by $N_1$ and $N_2$, respectively, and $\mathcal Y(\cdot,z)$ an intertwining operator of type ${M_3\choose M_1\,M_2}$. If $\mathcal Y(N_1, z)N_2=0$, then $\mathcal Y(\cdot, z)=0$ holds.
\end{proposition}

Let $I_V\binom{M_3}{M_1\  M_2}$ denote the space of all intertwining operators of type $\binom{M_3}{M_1 \ M_2}$, whose dimension is called the {\it fusion rule} of $M_1$, $M_2$ and $M_3$.
The tensor product (or \emph{fusion product}) of $M_1$ and $M_2$ in a $V$-module category $\mathcal{C}$, denoted by $M_1 \boxtimes M_2$, is an object in $\cC$ (unique up to isomorphisms) which represents the functor 
\begin{align*}
    \cC \rightarrow \mathrm{Vect}_\C,\quad M_3 \mapsto I_V\binom{M_3}{M_1\ M_2}.
\end{align*}

Note that there are natural isomorphisms 
\begin{align}\label{eq: skew symmetry}
    \Omega_r\colon I_V\binom{M_3}{M_1\ M_2}\xrightarrow{\simeq} I_V\binom{M_3}{M_2\ M_1},\quad (r \in \Z)
\end{align}
given by 
\begin{equation*}
\Omega_r(\mathcal Y)(v, z)u= e^{z'L_{-1}} \mathcal Y(u, z')v \Big\vert_{z'{}^n = e^{n(2r+1) \pi i} z^n\ (n \in \C)}.
\end{equation*}
Let $\cC$ be a category of $V$-modules containing $V$ and closed under tensor product. 
Then, under some technical assumptions in \cite{HLZ0}-\cite{HLZ2} on the composability of intertwining operators, $\cC$ has a vertex tensor category structure given by the point-wise tensor product $M_1 \boxtimes_{P(z)}M_2$, called the $P(z)$-tensor product, depending on the points $z \in\C \backslash \{0\}$ to evaluate the variable appearing in the intertwining operators.
It is noteworthy that we need different points to iterate the tensor product $M_1 \boxtimes_{P(z_1)}(M_2 \boxtimes_{P(z_2)} M_3)$.
However, the point-wise tensor products are all isomorphic by the functorial isomorphisms $M_1 \boxtimes_{P(z_1)}M_2 \simeq M_1 \boxtimes_{P(z_2)}M_2$, called the parallel transport. 
Then, by specializing $M_1 \boxtimes M_2 := M_1 \boxtimes_{P(1)}M_2$ and using the parallel transport, we obtain a braided tensor category structure on $\cC$. 
The unit $\one$ is $V$ itself with unit morphisms 
$$\ell_A\colon \one \boxtimes M\xrightarrow{\simeq} M,\quad r_A\colon M \boxtimes \one\xrightarrow{\simeq} M$$
induced by the structure map $Y(\cdot,z)\colon V\times M \rightarrow M(\!(z)\!)$ and $\Omega_1(Y)(\cdot,z)$. The associativity isomorphisms
$$\cA_{M_1, M_2, M_3}\colon M_1 \boxtimes (M_2\boxtimes M_3)\xrightarrow{\simeq} (M_1\boxtimes M_2) \boxtimes M_3$$
are induced by the associativity of $P(z)$-tensor products (or equivalently, intertwining operators). 
The braiding isomorphisms 
$$\cR_{M_1,M_2}\colon M_1 \boxtimes M_2 \xrightarrow{\simeq} M_2 \boxtimes M_1$$
are obtained from $\Omega_{r=1}$ in \eqref{eq: skew symmetry}. The double braidings 
$\cM_{M_1,M_2}=\cR_{M_2,M_1}\circ \cR_{M_1,M_2}$ are called the monodromy isomorphisms.
Hereafter, the braided tensor category structure on the categories $\cC$ of $V$-modules will mean the one obtained in this way.
When $V$ is $\Z$-graded, then $\cC$ admits a twist $\theta_M=e^{2\pi i L_0}|_M$, which satisfies the balancing isomorphism 
\begin{align}\label{eq: balancing axiom}
    \theta_{M\boxtimes N}=\cM_{M,N}\circ (\theta_M\boxtimes \theta_N).
\end{align} 
If $M$ has a left (and thus right) dual object, it is given by the contragredient module $M^*=\bigoplus_{\Delta}\Hom_\C(M_\Delta,\C)$. The category $\cC$ is called rigid if each object has a dual object, and thus ribbon if $V$ is $\Z$-graded.

\subsection{Commutative algebras}
Let $\cC$ be a vertex (and thus braided) tensor category of $V$-modules.
Vertex algebra extensions $(V,\omega)\subset (A,\omega)$ in $\cC$ preserving the conformal structure can be equivalently rephrased as commutative algebras in $\cC$ \cite{CKM, HKL}.

A commutative algebra in $\cC$ is a triple $(A, m, u)$ consisting of an object $A$ in $\cC$, and morphisms $m: A \otimes A \rightarrow A$, and $u: \one \rightarrow A$ which make the following diagrams commutative:
\begin{enumerate}[leftmargin=*]
    \item Associativity axiom
    \begin{equation*}
\begin{split}
\xymatrix{
A \boxtimes \left(A \boxtimes A \right)  \ar[rr]^(0.5){ \cA_{A, A, A} }\ar[d]_{ \Id_A \boxtimes m}
&& \left(A \boxtimes A\right)\boxtimes A  \ar[d]^{m \boxtimes \Id_A} &
\\
A \boxtimes A \ar[rd]_{m}
&& A\boxtimes A \ar[ld]^{m} &
\\
& A. &&
\\
}
\end{split}
\end{equation*}
\item Unit axiom
\begin{equation*}
    \xymatrix{
\one \boxtimes A \ar[r]^{\ell_A}\ar[d]_{u\boxtimes \Id_A} & A \ar[d]^{\Id_A} && A\boxtimes \one \ar[r]^{r_A}\ar[d]_{\Id_A \boxtimes u} & A \ar[d]^{\Id_A}
\\  
A\boxtimes A \ar[r]^m & A, && A\boxtimes A \ar[r]^m & A.
\\
}
\end{equation*}
\item Commutativity axiom
\begin{equation*}
\xymatrix{
A\boxtimes A \ar[rd]_{m}\ar[rr]^{\cR_{A, A}}
&& A\boxtimes A \ar[ld]^{m} &
\\
& A. &&
\\
}
\end{equation*}
\end{enumerate}
By restriction, the modules over $A$ as a conformal vertex algebra are also $V$-modules. The $A$-modules which lie in $\cC$ as $V$-modules can be rephrased as \emph{local $A$-modules} in $\cC_A$. 
With $A$ a commutative algebra in $\cC$, an $A$-module in $\cC$ is a pair $(M, m_M)$ of an object $M$ in $\cC$ and a morphism $m_M: A \boxtimes M \rightarrow M$ which make the following diagrams commutative:
\begin{enumerate}[leftmargin=*]
\item Associativity axiom
\begin{equation*}
\begin{split}
\xymatrix{
A \boxtimes (A\boxtimes M)    \ar[rr]^{  \cA_{A, A, M} } \ar[d]_{\Id_A \boxtimes m_M } 
&&  (A \boxtimes A) \boxtimes M  \ar[d]^{m \boxtimes \Id_M}
\\
A \boxtimes M \ar[rd]_{m_M} && A \boxtimes M \ar[ld]^{m_M} \\
& M. &} 
\end{split}
\end{equation*}
\item Unit axiom
\begin{equation*}
\xymatrix{
\one \boxtimes M \ar[rd]_{\ell_M}\ar[rr]^{ u\boxtimes \Id_M }
&& A \boxtimes M \ar[ld]^{m_M} &
\\
& M. &&
\\
}
\end{equation*}
\end{enumerate}
It is called local if it further makes the following diagram commute:
\begin{equation*}
\xymatrix{
A\boxtimes M \ar[rd]_{m_M}\ar[rr]^{\cM_{A,M}}
&& A\boxtimes M \ar[ld]^{m_M} &
\\
& M. &&
\\
}
\end{equation*}
Let $\cC_A$ (resp.\ $\cC_A^{\text{loc}}$) denote the category of $A$-modules (resp. local $A$-modules) in $\cC$ whose morphisms $f\colon (M,m_M) \rightarrow (N,m_N)$ are those $f \colon M \rightarrow N$ in $\cC$ which make the following diagram commute: 
\begin{equation*}
\xymatrix{
A \boxtimes  M    \ar[rr]^{  \Id_A \boxtimes f   } \ar[d]_{ m_M } 
&&  A \boxtimes N  \ar[d]^{m_N }
\\
M \ar[rr]^{f} && N.
} 
\end{equation*}
Note that the forgetful functor
\[
\cG_A: \cC_A \rightarrow \cC, \qquad \begin{array}{ccc}(M, m_M) &\mapsto& M,\\ f&\mapsto& f, \end{array} 
\]
admits a left adjoint, called the induction functor,
\[
\cF_A: \cC \rightarrow \cC_A, \qquad \begin{array}{ccc} M &\mapsto& (A \boxtimes M, m \boxtimes \Id_M),\\  f& \mapsto& \Id_A \boxtimes f. \end{array} 
\]
The natural isomorphism
\begin{equation}\label{eqn:fro-rec}
\Hom_{\cC_A}(\cF_A(M), N) \simeq \Hom_{\cC}(M, \cG_A(N))
\end{equation}
is called Frobenius reciprocity.
By the right unit constraint the diagram
\[
\xymatrix{
  (A\otimes \one) \otimes X \ar[rr]^{r_A \otimes \Id_X}\ar[d]_{\Id_A \otimes u \otimes \Id_X} && A \otimes X \ar[d]^{\Id_A \otimes \Id_X}
\\  
 (A\otimes A) \otimes X \ar[rr]^m && A \otimes X 
\\
}
\]
commutes, i.e. the multiplication on $\cF_A(X)$ restricted to $X$ is surjective and so $\cF_A(X)$ is generated by $X$ as an $A$-module. 
\begin{theorem}[\cite{CKM}]\label{sum1} 
Let $A$ be a commutative algebra in a vertex (braided) tensor category $\cC$ of $V$-modules.
\begin{enumerate}[leftmargin=*]
\item $\cC_A$ has a tensor category structure and $\cC_A^{loc}$ has a braided tensor category structure which is a tensor subcategory of $\cC_A$. 
\item $\cF_A: \cC \rightarrow \cC_A$ is a tensor functor and restricted to a braided tensor functor $\cF_A: \cC^0 \rightarrow \cC_A^{loc}$
from a full subcategory $\cC^0\subset \cC$ consisting of objects $M$ satisfying $\cM_{A, M}=\Id_{A\btimes M}$.
\item The category of modules in $\cC$ over $A$ as a vertex algebra admits a
braided tensor category structure, which is canonically braided tensor equivalent to the category $\cC_A^{loc}$ of local $\cC$-algebra modules. 
\end{enumerate}
\end{theorem}
This theorem is generalized to direct limit completions of $\cC$ \cite{CMY} and commutative superalgebra objects corresponding to vertex superalgebra extensions \cite{CKM}.

\subsection{\tmath{$C_1$}-cofinite modules}
Typically, we obtain a braided tensor category $\cC$ of $V$-modules by considering $C_1$-cofinite $V$-modules. Recall that a $V$-module $M$ is called $C_1$-cofinite if it satisfies $\dim M/V^+_{(-1)}M<\infty$, where $V^+=\bigoplus_{\Delta>0}V_\Delta$.
We denote by $V\mod_{C_1}$ the category of $C_1$-cofinite $V$-modules. 
Note that $V\mod_{C_1}$ is closed under taking quotients and extensions, but might not be under taking submodules, i.e., $V\mod_{C_1}$ is just additive.

\begin{ex}[\cite{KL}]\label{ex: affine Kazhdan-Lusztig} 
\textup{For the affine vertex algebras $V^\kk(\g)$ at levels $\kk\in \C\backslash \Q$, simple modules $M$ are $C_1$-cofinite if and only if $M\simeq \weyl_\lambda^\kk$ for some $\lambda \in P_+$. Since $\mathrm{Ext}^1(\weyl_\lambda^\kk,\weyl_\mu^\kk)=0$ for $\lambda,\mu \in P_+$ at $\kk\in \C\backslash \Q$, $V^\kk(\g)\mod_{C_1}$ is abelian and agrees with the Kazhdan-Lusztig category $\KL^\kk(\g)$. This category is semisimple and rigid. The dual of $\weyl_\lambda^\kk$ is $\weyl_{\lambda^*}^\kk$ which is induced by the dual $\g$-module $L(\lambda^*)$ of $L(\lambda)$.
Moreover, the fusion rules 
$$\weyl_\lambda^\kk\boxtimes \weyl_\mu^\kk\simeq \bigoplus_{\nu\in P_+}c_{\lambda,\mu}^\nu \weyl_\nu^\kk$$
agree with the tensor product decomposition of $\g$-modules: 
$$L(\lambda)\otimes L(\mu)\simeq \bigoplus_{\nu \in P_+} c_{\lambda,\mu}^\nu L(\nu).$$
}\end{ex}

\begin{theorem}[{\cite[Theorem 1.2.]{H1}}] \label{thm: Huang C1}
The category $V\mod_{C_1}$ of $C_1$-cofinite modules over simple conformal vertex algebras $V$ satisfies the associativity of intertwining operators. 
In particular, $V\mod_{C_1}$ has a braided monoidal category structure with a twist.
\end{theorem}
Note that the algebras of chiral differential operators $\cdo{G}{\kk}$ at levels $\kk \in \C \backslash \Q$ in \S \ref{sec: CDOs} are conformal extensions of $V^\kk(\g)\otimes V^{\kk_o}(\g)$. However, it is not a commutative algebra object in the Deligne tensor product $\KL^\kk(\g)\boxtimes \KL^{\kk_o}(\g)$, but rather in the direct limit completion $\Ind(\KL^\kk(\g)\boxtimes \KL^{\kk_o}(\g))$. 
In \cite{CMY}, it was proven that given a vertex (and thus braided) tensor category of $V$-modules $\sC$, the direct limit completion $\Ind(\sC)$ in the category of (arbitrary) $V$-modules has a natural structure of a vertex (and thus braided) tensor category. 
As Theorem \ref{thm: Huang C1} indicates, the assumption on $\sC$ of being an abelian category is not necessary.
In this generality, let $\mathrm{Ind}(V\mod_{C_1})$ denote the subcategory of $V\mod$ consisting of unions of $C_1$-cofinite $V$-modules.

\begin{theorem}[{\cite[Theorem 6.10]{MN}}]
The completion $\mathrm{Ind}(V\mod_{C_1})$ of $C_1$-cofinite modules over simple conformal vertex algebras $V$ is additive, co-complete, and has a braided monoidal category structure with a twist so that the embedding 
$$V\mod_{C_1} \rightarrow \mathrm{Ind}(V\mod_{C_1})$$
is a braided monoidal functor preserving twists.
\end{theorem}

\begin{comment}
$C_1$-codimension also gives us a bound on fusion rules.
\begin{theorem}\cite[Theorem 3.5.]{H1} \label{thm:dimC1} Let $M_1, M_2$ be $\mathbb N$-gradable weak $V$-modules. Then for a weak surjective
product $M_3$ of $M_1$ and $M_2$,
\[
\text{dim}(M_3/C_1(M_3))  \leq \text{dim}(M_1/C_1(M_1))\text{dim}(M_2/C_1(M_2)).
\]
\end{theorem}
\end{comment}

\subsection{Mirror equivalence}

In order to construct an equivalence between $\KL^k(\g)$ and $\KL^k_f(\g)$, we will use the mirror equivalence \cite{CKM2, McR} obtained from kernel objects, which will be realized as conformal vertex algebra extensions. 

Let $U$ and $V$ be simple, self-contragredient, conformal vertex algebras and $A$ a simple conformal vertex algebra extension of $U\otimes V$.
Suppose the condition (M):
\begin{itemize}
    \item[(M1)] $U$ and $V$ form a dual pair in $A$, i.e.,
    $$\Com(U,A)=V,\quad \Com(V,A)=U.$$
    \item[(M2)] $A$ is semisimple as a $U\otimes V$-module:
    \begin{align*}
    A\simeq \bigoplus_{i \in I} U_i\otimes V_i,
\end{align*}
where $U_i$ is a distinct simple $V$-module with finite-semisimple $V$-modules $V_i$. 
Let $0\in I$ and $U_0=U$, $V_0=V$.
    \item[(M3)] The $U$-modules $U_i$ $(i\in I)$ lie in a braided tensor category $\cC_U$ of $V$-modules which is finite-semisimple, rigid and closed under contragredients.
    \item[(M4)] The $V$-modules $V_i$ $(i\in I)$ lie in a braided tensor category $\cC_V$ which is closed under contragredients.
\end{itemize}

Let us denote by $\cC_{U\otimes V}$ the category of $U\otimes V$-modules consisting of finite direct sums of $M\otimes N$ for $M \in \cC_U$ and $N \in \cC_V$.
Let 
$$\cC_U(A) \subset \cC_U,\quad \cC_V(A)\subset \cC_V$$
the subcategories consisting of finite direct sums of $U_i$ and $V_i$ ($i \in I$), respectively.
\begin{theorem}[{\cite{McR}}]\label{thm:main_thm}
Let $U\otimes V \subset A$ be a conformal extension of simple vertex algebras as above satisfying the condition \textup{(M)}. 
If every intertwining operator 
$$\mathcal Y(\cdot, z)\colon M_1\times M_2 \rightarrow M_3\{z\}[\log z]$$
of $U\otimes V$-modules $M_1, M_2 \in \cC_{U\otimes V}$ and $M_3 \in \mathrm{Ind}(\cC_{U\otimes V})$ factors through a submodule $M_3'\subset M_3$ in $\cC_{U\otimes V}$, then $\cC_U(A)$ is a braided tensor subcategory of $\cC_U$ and $\cC_V(A)$ is a semisimple ribbon braided tensor subcategory of $\cC_V$ with distinct simple objects $V_i$ $(i \in I)$. Moreover, there is a braid-reversed tensor equivalence 
$$\tau: \cC_U(A)\rightarrow\cC_V(A),\quad U_i \mapsto V_i^*.$$
\end{theorem}

\section{Principal nilpotent case}
\subsection{\tmath{$C_1$}-cofiniteness}
\begin{proposition}\label{prop:C_1}
Let $\g$ be a simple Lie algebra and $f$ a nilpotent element.
The $\W^\kk(\g,f)$-modules $\weyl^\kk_{\lambda,f}$ with $\lambda \in P_+$ are $C_1$-cofinite for all levels $\kk\in \C$.
\end{proposition}
\begin{proof}
The assertion is essentially shown in the proof of \cite[Theorem 6.2]{KW}.
Indeed, $\weyl^\kk_{\lambda,f}$ is realized as the kernel of the differential $d$ of the subcomplex $C_{\lambda,f}^{-,\bullet}:=\mathscr{C}^{-,\bullet}\otimes L(\lambda)$ in the BRST complex $C_f^\bullet(\weyl^\kk_{\lambda})$ (see Section \ref{sec: W-algebras}) with $P_0=L(\lambda)$ in \emph{loc.\ cit.}. The complex $C_{\lambda,f}^{-,\bullet}$ has a bicomplex structure $d=d_1^P+d_2^P$ with a convergent spectral sequence such that $E_2^{p,q}=H^p(H^q(C_{\lambda,f}^{-,\bullet},d_1^P),d_2^P)\simeq \delta_{p,0},\delta_{q,0}\W^{\kk}(\g,f)\otimes L(\lambda)$ by the proof of \cite[Theorem 4.1]{KW}. 
Now, the assertion follows from  $\weyl^\kk_{\lambda,f}/\W^\kk(\g,f)^+_{(-1)}\weyl^\kk_{\lambda,f}\simeq L(\lambda)$ as vector spaces.
\end{proof}
\begin{remark}
\textup{The same proof applies for  $\W$-superalgebras associated with basic-classical Lie superalgebras.}
\end{remark}
\begin{proposition}\label{prop: BRST reduction for intertwining operators}
Let $\g$ be a simple Lie algebra and $f$ a nilpotent element.
The BRST reduction \eqref{eq: BRST reduction} is naturally extended to a functorial homomorphism between the spaces of intertwining operators. 
\begin{align*}
    H_f^0\colon I_{V^\kk(\g)}\binom{M_3}{M_1\, M_2} \rightarrow I_{\W^\kk(\g,f)}\binom{H_f^0(M_3)}{H_f^0(M_1)\ H_f^0(M_2)},
\end{align*}
which is injective on $\KL^\kk(\g)$ at levels $\kk \in \C \backslash \Q$.
\end{proposition}
\begin{proof}
Let $\mathcal{Y}(\cdot,z)\colon M_1\times M_2 \rightarrow M_3 \{z\}[\log z]$ be an intertwining operator. 
It naturally extends to the BRST complexes
$$\mathcal{Y}(\cdot,z)\otimes Y(\cdot,z)\colon C_f^\bullet(M_1)\times C_f^\bullet(M_2) \rightarrow C_f^\bullet(M_3) \{z\}[\log z]$$
by tensoring with a structure map of vertex (super)algebras appearing in the d.g. vertex algebra $C_f^\bullet(V^\kk(\g))$. 
Thanks to the Jacobi identity of the intertwining operators, $\widetilde{\mathcal{Y}}(\cdot,z):=\mathcal{Y}(\cdot,z)\otimes Y(\cdot,z)$ satisfies $[d,\widetilde{\mathcal{Y}}(u,z)]=\widetilde{\mathcal{Y}}(du,z)$ where $d=Q_{(0)}$ denotes the differential of the BRST complexes, see \cite{KW} for the definition of $Q$. Now, it is straightforward to show that $\widetilde{\mathcal{Y}}(\cdot,z)$ defines an intertwining operator on the cohomologies
$$[\widetilde{\mathcal{Y}}](\cdot,z)\colon H_f^\bullet(M_1)\times H_f^\bullet(M_2) \rightarrow H_f^\bullet(M_3) \{z\}[\log z].$$
This completes the first assertion.
To show the second assertion, we may take $M_1=\weyl^\kk_{\lambda}$, $M_2=\weyl^\kk_{\mu}$ and $M_3=\weyl^\kk_{\nu}$, and consider a non-zero intertwining operator $\mathcal{Y}_\iota(\cdot,z)\in I_{V^\kk(\g)}\binom{M_3}{M_1\ M_2}$ induced from a non-zero homomorphism $\iota\in \Hom_\g(L(\lambda)\otimes L(\mu),L(\nu))$ by using Example \ref{ex: affine Kazhdan-Lusztig}. This induces an intertwining operator $[\widetilde{\mathcal{Y}}_\iota](\cdot,z) \in I_{\W^\kk(\g,f)}\binom{\weyl^\kk_{\nu,f}}{\weyl^\kk_{\lambda,f}\ \weyl^\kk_{\mu,f}}$.
As in the proof of Proposition \ref{prop:C_1}, the cohomologies are realized as the subspaces $\ker d \subset C_{\ast,f}^{-,\bullet}$ for $\ast=\lambda,\mu,\nu$. 
By the shape $C_{\ast,f}^{-,\bullet}=\mathscr{C}^{-,\bullet}\otimes L(\ast)$, the composition 
$$C_{\lambda,f}^{-,0}\times C_{\mu,f}^{-,0} \hookrightarrow C_{f}^0(\weyl^\kk_\lambda)\times C_{f}^0(\mu^\kk_\lambda) \overset{\widetilde{\mathcal{Y}}_\iota}{\longrightarrow} C_{f}^{0}(\weyl^\kk_\nu)\{z\}[\log z]\twoheadrightarrow C_{\nu,f}^{-,0}\{z\}[\log z]$$ 
is clearly non-zero since it restricts to a non-zero homomorphism $\iota\colon L(\lambda)\times L(\mu)\rightarrow L(\nu)$ up to the factor $z^\alpha (\log z)^n$. This further induces a non-zero homomorphism on the cohomology as $\weyl^\kk_{\ast,f}/\W^\kk(\g,f)^+_{(-1)}\weyl^\kk_{\ast,f}\simeq L(\ast)$ as vector spaces for $\ast=\lambda,\mu,\nu$.
This completes the proof.
\end{proof}

From now on, let $\g$ be simply-laced, i.e., of type $ADE$. We identify $P$ and $\check{P}$ so that we denote the $\W^\kk(\g)$-modules $T^\kk_{\lambda,\mu}$ with $\lambda, \mu \in P_+$ instead of $T^\kk_{\lambda,\check{\mu}}$. 

\begin{proposition}\label{Main A}
Let $\g$ be of type $ADE$ and $\kk\in \C\setminus \Q$.
\begin{enumerate}[leftmargin=*]
    \item The $\W^\kk(\g)$-modules $T^\kk_{\lambda, \mu}$ $(\lambda,\mu \in P_+)$ are $C_1$-cofinite.
    \item $I_{\W^\kk(\g)}\binom{T^\kk_{\lambda, \mu}}{T^\kk_{\lambda, 0}\, T^\kk_{0, \mu}}\neq 0$ for $\lambda,\mu \in P_+$.
\end{enumerate}
\end{proposition}
Before giving the proof, let $\kk^o\in \C\backslash \Q$ be the level determined by $\frac{1}{\kk^o+h^\vee}+\frac{1}{\kk+h^\vee}=1$ as in \eqref{eq: shifting the level} and introduce the $V^{\kk^o-1}(\g)\otimes V_Q$-modules 
\begin{align}\label{eq: M-modules}
M^\kk_{\lambda,\mu}=\weyl_\lambda^{\kk^o-1}\otimes V_{\mu+Q},\quad (\lambda,\mu \in P_+).    
\end{align}

\begin{proof}[Proof of Proposition \ref{Main A}]
By Theorem \ref{thm:GKO1}, we have the $V^{\kk^o}(\g)\otimes \W^\kk(\g)$-submodules
\begin{align*}
    \weyl_{\lambda}^{\kk^o}\otimes T_{\lambda,0}^\kk \subset M^\kk_{0,\lambda},\quad \weyl_{0}^{\kk^o}\otimes T_{0,\mu}^\kk \subset M^\kk_{\mu,0}.
\end{align*}
Since $M^\kk_{0,\lambda} \boxtimes M^\kk_{\mu,0}\simeq M^\kk_{\mu,\lambda}$ as $V^{\kk^o-1}(\g)\otimes V_Q$-modules, there exists a corresponding canonical intertwining operator $\io \colon M^\kk_{0,\lambda} \otimes M^\kk_{\mu,0}\rightarrow M^\kk_{\mu,\lambda}\{z\}[\log z]$.
As $M^\kk_{0,\lambda}$ and $M^\kk_{\mu,0}$ are simple modules, the restriction of $\io$ on $(\weyl_{\lambda}^{\kk^o}\otimes T_{\lambda,0}^\kk) \otimes (\weyl_{0}^{\kk^o}\otimes T_{0,\mu}^\kk)$ is nonzero by Proposition \ref{prop:intertwiner}. As $\io$ is an intertwining operator of $V^{\kk^o-1}(\g)\otimes V_Q$-modules, the restriction factors as 
$$\io\colon (\weyl_{\lambda}^{\kk^o}\otimes T_{\lambda,0}^\kk) \otimes (\weyl_{0}^{\kk^o}\otimes T_{0,\mu}^\kk)\rightarrow (\weyl_{\lambda}^{\kk^o}\otimes T_{\lambda,\mu}^\kk)\{z\}[\log z].$$
As $(\weyl_{\lambda}^{\kk^o}\otimes T_{\lambda,\mu}^\kk)$ is simple, $\io$ is surjective, which implies (1) by Theorem \ref{thm: Huang C1}. 
Since $I_{V^{\kk^o}(\g)}\binom{\weyl_{\lambda}^{\kk^o}}{\weyl_{\lambda}^{\kk^o},\weyl_{0}^{\kk^o}}\simeq \C$ spanned by $\Omega_1(Y)(\cdot,z)$, $\io=\Omega_1(Y)(\cdot,z)\otimes \mathcal{Y}^\W(\cdot,z)$ for some intertwining operator $\mathcal{Y}^\W(\cdot,z)$ of $\W^\kk(\g)$-modules. 
As $\io\neq0$, so is $\mathcal{Y}^\W(\cdot,z)$ and thus (2) follows.
\end{proof}
Let us now consider the properties of $\W^\kk(\g)$-modules $T^\kk_{\lambda,\mu}$ in the braided monoidal category $\W^\kk(\g)\Mod_{C_1}$. 
By Theorem \ref{thm:GKO1}, we regard $\W^{\kk-1}(\g)\otimes V_Q$ as a commutative algebra object in the completion of the braided monoidal category $(\W^{\kk}(\g)\otimes \W^{\kk^o}(\g))\mod_{C_1}$.
\subsection{Centralizing property}
\begin{proposition}\label{prop:monodromy}
Let $\g$ be of type $ADE$ and $\kk\in \C\setminus \Q$. 
Then, $T^{\kk}_{\lambda, 0}$ and $T^{\kk}_{0, \mu}$ projectively centralize each other for $\lambda, \mu \in P_+$. More precisely, 
\begin{align}\label{eq: proj centralizer}
    \cM_{T^{\kk}_{\lambda, 0},T^{\kk}_{0, \mu}} = e^{-2\pi i (\lambda,\mu)} \mathrm{Id}_{T^{\kk}_{\lambda, 0}\boxtimes T^{\kk}_{0, \mu}}.
\end{align}
\end{proposition}
\begin{proof}
  We  consider $\kk'=\kk-1$ instead of $\kk$. 
  Let us set 
  $$\bM_{\lambda,\mu}=T_{\lambda,\mu}^{\kk'}\otimes V_{Q+\lambda^*+\mu^*},\quad \lambda,\mu \in P_+.$$ 
  Let $\io\colon \bM_{\lambda,0}\otimes \bM_{0,\mu}\rightarrow (\bM_{\lambda,0}\boxtimes \bM_{0,\mu})\{z\}[\log z]$ be the canonical intertwining operator of $\W^{\kk'}(\g)\otimes V_Q$-modules. Note that $\bM_{\lambda,0}\boxtimes \bM_{0,\mu}\simeq (T_{\lambda,0}^{\kk'}\boxtimes T_{0,\mu}^{\kk'}) \otimes V_{Q+\lambda^*+\mu^*}$.
  By Corollary \ref{cor:urod}, $\io$ restricts as follows:
    \begin{center}
\begin{tikzcd}[ column sep = small]
			\bM_{\lambda,0}\otimes \bM_{0,\mu}
			\arrow[r, " \io"]&
			(\bM_{\lambda,0}\boxtimes \bM_{0,\mu})\{z\}[\log z]\\
		      (T_{0,0}^\kk\otimes T^{\kk^o}_{0,\lambda})\otimes (T_{0,\mu}^\kk \otimes T_{0,0}^{\kk^o})
			\arrow[r, " "] \arrow[u,symbol=\subset, " "]&
			 (T_{0,\mu}^\kk\otimes T^{\kk^o}_{\lambda,0})\{z\}[\log z]\arrow[u,symbol=\subset, " "].
\end{tikzcd}
\end{center}
As $\bM_{\lambda,0}$ and $\bM_{0.\mu}$ are simple, the restriction of $\io$ on the bottom row is nonzero by Proposition \ref{prop:intertwiner} and thus surjective as $T_{0,\mu}^\kk\otimes T^{\kk^o}_{\lambda,0}$ is simple by Theorem \ref{thm: simplicity of T-modules}. Since $\io$ on the top row is surjective
and since $\bM_{\lambda,0}$ and  $\bM_{0,\mu}$ are generated by $T_{0,0}^\kk\otimes T^{\kk^o}_{0,\lambda}$ and $T_{0,\mu}^\kk \otimes T_{0,0}^{\kk^o}$, $\bM_{\lambda,0}\boxtimes \bM_{0,\mu}$ is generated by $T_{0,\mu}^\kk\otimes T^{\kk^o}_{\lambda,0}$ as a $\W^{\kk'}(\g)\otimes V_Q$-module. Since $L_0$ acts on $T_{0,\mu}^\kk\otimes T^{\kk^o}_{\lambda,0}$ semisimply, it must do so as well on $\bM_{\lambda,0}\boxtimes \bM_{0,\mu}$ and thus on $T^{\kk}_{\lambda, 0}\boxtimes T^{\kk}_{0, \mu}$. 
Since $\W^{\kk'}(\g)\otimes V_Q$ is a $\Z$-graded vertex algebra, the monodromy can be computed by using twists \eqref{eq: balancing axiom}:
\begin{align*}
&\theta_{\bM_{\lambda,0}\boxtimes\bM_{0,\mu}}=\cM_{\bM_{\lambda,0},\bM_{0,\mu}}\circ (\theta_{\bM_{\lambda,0}}\boxtimes \theta_{\bM_{0,\mu}}),\\
&\cM_{\bM_{\lambda,0},\bM_{0,\mu}}=\cM_{T^{\kk'}_{\lambda, 0},T^{\kk'}_{0, \mu}}\otimes \cM_{V_{Q+\lambda^*},V_{Q+\mu^*}}.
\end{align*}
By $\theta_M=e^{2\pi i L_0}|_M$ in general and \eqref{eq:confweights}, the first equation implies $\cM_{\bM_{\lambda,0},\bM_{0,\mu}}=\Id$. 
Then $\cM_{V_{Q+\lambda^*},V_{Q+\mu^*}}=e^{2\pi i (\lambda,\mu)}\Id$ in the second equality implies the assertion.
\end{proof}
\subsection{Main result}\label{sec: Main results for pincipal case}
Now, we are ready to show the following main result on fusion rules of modules of principal $\W$-algebras.
\begin{theorem}\label{thm:fusion, regular case}
    For $\g$ of type $ADE$ and $\kk\in \C\setminus \Q$, 
    $$T_{\lambda,\lambda'}^\kk\boxtimes T_{\mu,\mu'}^\kk\simeq \bigoplus_{\nu,\nu'\in P_+}c_{\lambda,\mu}^\nu c_{\lambda',\mu'}^{\nu'} T_{\nu,\nu'}^\kk$$
    holds as $\W^\kk(\g)$-modules for $\lambda,\lambda'\in P_+$ and $\mu,\mu'\in P_+$.
\end{theorem}
\noindent 
Here, $c_{\lambda,\mu}^\nu$ are the branching rules of tensor products of $\g$-modules, see Example \ref{ex: affine Kazhdan-Lusztig}. 

Let $\W^{\kk^o}(\g)\Mod_{C_1}^{0}$ be the subcategory of $\W^{\kk^o}(\g)\Mod_{C_1}$ that centralizes all $T^{\kk^o}_{\lambda, 0}$ with $\lambda \in P_+ \cap Q$. (See also Proposition \ref{prop:monodromy}.)
The key ingredient of the proof is that  the induction functor for the extension 
$$V^\kk(\g)\otimes \W^{\kk^o}(\g)\subset A:=V^{\kk'}(\g)\otimes V_Q,$$ 
given by\footnote{From now on we omit the algebra multiplication when using the induction functor.} 
\begin{align}\label{eq: induction functor}
\begin{array}{cccc}
    \cF_V\colon &\W^{\kk^o}(\g)\Mod_{C_1}^{0}&\rightarrow& \KL^{\kk'}(\g)\boxtimes V_Q\Mod\\ &M&\mapsto &A\boxtimes (V^\kk(\g)\otimes M)
\end{array}
\end{align}
where $\kk'=\kk-1$ has a left inverse given by the coset functor:
\begin{align*}
\begin{array}{cccc}
    \Com_V \colon &\KL^{\kk'}(\g)\boxtimes V_Q\Mod &\rightarrow &\W^{\kk^o}(\g)\Mod_{C_1}^{0}\\ &N&\mapsto & \Hom_{V^\kk(\g)}(V^\kk(\g),N).
\end{array}
\end{align*}
\begin{proposition}\label{prop: left inverse by coset}
$\Com_V(\cF_V(M))\simeq M$ holds for $M$ in $\W^{\kk^o}(\g)\Mod_{C_1}^{0}$.
\end{proposition}
\begin{proof}
By Theorem \ref{thm:GKO1}, we have 
\begin{align*}
    \cF_V(M) 
    &= (V^{\kk'}(\g)\otimes V_Q) \boxtimes (V^\kk(\g)\otimes M)\\
    &\simeq \bigoplus_{\mu \in P_+\cap Q} \weyl_\mu^\kk\otimes (T^{\kk^o}_{\mu,0}\boxtimes M).
\end{align*}
Hence, $\Com_V(\cF_V(M))\simeq T^{\kk^o}_{0,0}\boxtimes M \simeq M$.
\end{proof}
Let us write $\W^\kk(\g)$ as $\W^\kk$, $V^{\kk}(\g)$ as $V^{\kk}$, and $V_{Q+\lambda}$ as $V_\lambda$ for simplicity.
\begin{proposition}\label{prop: identification of induction}
    For $\lambda \in P_+$, there is an isomorphism of $V^{\kk'}\otimes V_Q$-modules
    $$\cF_V(T_{0,\lambda}^{\kk^o})\simeq \weyl_\lambda^{\kk'}\otimes V_{\lambda^*}.$$
\end{proposition}
\begin{proof}
    Recall that $\KL^{\kk'}(\g)\boxtimes V_Q\Mod $ is semisimple with simple objects $\weyl_{\mu}^{\kk'}\otimes V_{\nu}$. Then it follows from Frobenius reciprocity and Corollary \ref{cor:urod} that  
\begin{align*}
    \Hom_{V^{\kk'}\otimes V_Q}(\cF_V(T_{0,\lambda}^{\kk^o}),\weyl_{\mu}^{\kk'}\otimes V_{\nu}) &\simeq \Hom_{\W^{\kk}\otimes \W^{\kk^o}}(V^{\kk}\otimes T_{0,\lambda}^{\kk^o},\weyl_{\mu}^{\kk'}\otimes V_{\nu})\\
    &\simeq \bigoplus_{\mu'\equiv \mu+\nu}\Hom_{\W^{\kk}\otimes \W^{\kk^o}}(V^{\kk}\otimes T_{0,\lambda}^{\kk^o},\weyl_{\mu'}^{\kk}\otimes T^{\kk^o}_{\mu',\mu})\\
    &\simeq \bigoplus_{\mu'\equiv \mu+\nu}\delta_{0,\mu'}\delta_{\lambda,\mu} \Id
\end{align*}
and, hence, $\cF_V(T_{0,\lambda}^{\kk^o})\simeq \weyl_\lambda^{\kk'}\otimes V_{\lambda^*}$.
\end{proof}
Let us introduce the induction functor for the extension 
$$\W^\kk(\g,f)\otimes \W^{\kk^o}(\g)\subset A_f:=\W^{\kk'}(\g,f)\otimes V_Q,$$ 
given by 
\begin{align}\label{eq: induction functor 2}
\begin{array}{cccc}
    \cF_f\colon &(\W^{\kk}(\g,f)\otimes\W^{\kk^o}(\g))\Mod_{C_1}^{0}&\rightarrow& \W^{\kk'}(\g,f)\Mod\boxtimes V_Q\Mod\\ &M&\mapsto &A_f\boxtimes M.
\end{array}
\end{align}
\begin{proof}[Proof of Theorem \ref{thm:fusion, regular case}]
By the associativity of tensor products and the duality in Theorem \ref{thm: simplicity of T-modules}, it suffices to show 
\begin{align}\label{eq: fusions to show}
    T_{0,\lambda}^\kk\boxtimes T_{0,\mu}^\kk\simeq \bigoplus_{\nu\in P_+}c_{\lambda,\mu}^\nu T_{0,\nu}^\kk,\quad T_{\lambda,0}^\kk\boxtimes T_{0,\mu}^\kk\simeq T_{\lambda,\mu}^\kk.
\end{align}
Thanks to Proposition \ref{prop: left inverse by coset}, the first isomorphism of \eqref{eq: fusions to show} follows from 
\begin{align*}
    \cF_V(T_{0,\lambda}^{\kk^o}\boxtimes T_{0,\mu}^{\kk^o})
    &\simeq \cF_V(T_{0,\lambda}^{\kk^o})\boxtimes \cF_V(T_{0,\mu}^{\kk^o})\\
    &\simeq (\weyl_{\lambda}^{\kk'}\otimes V_{\lambda^*})\boxtimes (\weyl_{\mu}^{\kk'}\otimes V_{\mu^*})\\
    &\simeq \bigoplus_{\nu \in P_+}c_{\lambda,\mu}^\nu \cF_V(T_{0,\nu}^{\kk^o})
\end{align*}
after replacing $\kk^o$ with $\kk$, as the functor $\cF_V$ is monoidal and $V^\kk(\g)\otimes T^{\kk^o}_{0,\lambda'}$ with $\lambda\in P_+$ are local modules for $A$ by Proposition \ref{prop:monodromy}. 
Next, we show the second isomorphism.
Note that the case $\lambda \in P_+ \cap Q$ and $\mu \in P_+$ follows from Proposition \ref{prop: identification of induction} and Theorem \ref{thm:GKO1} as 
\begin{align*}
    \cF_V(T^{\kk^o}_{0,\mu})&\simeq \bigoplus_{\lambda \in P_+\cap Q} (\weyl^\kk_\lambda\otimes T^{\kk^o}_{\lambda,0})\boxtimes (V^\kk\otimes T^{\kk^o}_{0,\mu})\\
    &\simeq \bigoplus_{\lambda \in P_+\cap Q} \weyl^\kk_\lambda\otimes (T^{\kk^o}_{\lambda,0}\boxtimes T^{\kk^o}_{0,\mu}).
\end{align*}
and thus $T^{\kk^o}_{\lambda,0}\boxtimes T^{\kk^o}_{0,\mu}\simeq T^{\kk^o}_{\lambda,\mu}$, by replacing $\kk^o$ with $\kk$. 
Now, we have 
\begin{align}\label{eq: identification of induction for W}
    A_f \boxtimes (T_{0,\mu}^{\kk}\otimes T_{0,\lambda}^{\kk^o})\simeq T_{\lambda,\mu}^{\kk'} \otimes V_{\lambda^*+\mu^*}
\end{align}
for $\lambda,\mu \in P_+$ when $f$ is principal.
Indeed, Corollary \ref{cor:urod} implies 
\begin{align*}
    A_f\boxtimes (T_{0,\mu}^{\kk}\otimes T_{0,\lambda}^{\kk^o}) 
    &\simeq \bigoplus_{\mu'\in P_+\cap Q} (T_{\mu',0}^{\kk}\otimes T_{\mu',0}^{\kk^o}) \boxtimes (T_{0,\mu}^{\kk}\otimes T_{0,\lambda}^{\kk^o})\\
    &\simeq \bigoplus_{\mu'\in P_+\cap Q} T_{\mu',\mu}^{\kk}\otimes T_{\mu',\lambda}^{\kk^o}\\
    &\simeq T_{\lambda,\mu}^{\kk'} \otimes V_{\lambda^*+\mu^*}
\end{align*}
as $\W^{\kk'}\otimes \W^{\kk^o}$-modules. Then, the isomorphism \eqref{eq: identification of induction for W} as $\W^{\kk'}\otimes V_Q$-modules follows from the simplicity of $T_{\lambda,\mu}^{\kk'} \otimes V_{\lambda^*+\mu^*}$ and the existence of a non-zero homomorphism by Frobenius reciprocity
\begin{align*}
    \Hom_{\W^{\kk'}\otimes V_Q}&(A_f \boxtimes (T_{0,\mu}^{\kk}\otimes T_{0,\lambda}^{\kk^o}), T_{\lambda,\mu}^{\kk'} \otimes V_{\lambda^*+\mu^*})\\
    &\simeq \Hom_{\W^{\kk}\otimes \W^{\kk^o}}(T_{0,\mu}^{\kk}\otimes T_{0,\lambda}^{\kk^o}, T_{\lambda,\mu}^{\kk'} \otimes V_{\lambda^*+\mu^*})\simeq \C \Id.
\end{align*}
This completes the proof.
\end{proof}

\section{General nilpotent case}
\subsection{Main result}
\begin{theorem}\label{thm:fusion}
    Let $\g$ be of type $ADE$ and $\kappa \in \C \backslash \Q$. 
    \begin{enumerate}[leftmargin=*]
    \item For $\lambda \in P_+$, $\weyl_{\lambda,f}^\kk$ are simple $\W^\kk(\g,f)$-modules.
    \item For $\lambda,\mu \in P_+$,
    \begin{align*}
        \weyl_{\lambda,f}^\kk \boxtimes \weyl_{\mu,f}^\kk  \cong \bigoplus_{\nu \in P_+} c_{\lambda,\mu}^\nu \weyl_{\nu,f}^\kk.
    \end{align*}
    \end{enumerate}   
\end{theorem}
\begin{proof}
We may show the assertions for $\kk'=\kk-1$ instead of $\kk$.

\noindent
(1) Let us write $\W^\kk(\g,f)$ as $\W^\kk_f$ and, in particular, $\W^\kk$ when $f$ is principal as before. Consider $\cF_f(\W^\kk_f\otimes T^{\kk^o}_{0,\lambda})$ as in \S \ref{sec: Main results for pincipal case}. 
By Corollary \ref{cor:urod}, 
\begin{align}\label{eq: decomposition in the urod case}
    \cF_f(\W^\kk_f\otimes T^{\kk^o}_{0,\lambda})\simeq \bigoplus_{\mu \in P_+\cap Q}\weyl_{f,\lambda}^\kk \otimes T_{\mu,\lambda}^{\kk^o}
\end{align}
as $\W^\kk_f\otimes \W^{\kk^o}$-modules. 
Let us show first that $\cF_f(\W^\kk_f\otimes T^{\kk^o}_{0,\lambda})$ is a simple $\W^{\kk'}_f\otimes V_Q$-module. 
Otherwise, $\cF_f(\W^\kk_f\otimes T^{\kk^o}_{0,\lambda})$ admits a short exact sequence
\begin{align*}
    0 \rightarrow N \rightarrow \cF_f(\W^\kk_f\otimes T^{\kk^o}_{0,\lambda}) \rightarrow M \rightarrow 0
\end{align*}
for some proper submodule $N$. Since $\cF_f(\W^\kk_f\otimes T^{\kk^o}_{0,\lambda})$ is generated by $\W^\kk_f\otimes T^{\kk^o}_{0,\lambda}$, it follows that $N \cap (\W^\kk_f\otimes T^{\kk^o}_{0,\lambda^*})=0$ and thus $\W^\kk_f\otimes T^{\kk^o}_{0,\lambda}\subset M$. By dualizing the sequence, we obtain 
\begin{align*}
    0 \rightarrow M^* \rightarrow \cF_f(\W^\kk_f\otimes T^{\kk^o}_{0,\lambda^*}) \rightarrow N^* \rightarrow 0.
\end{align*}
It follows from \eqref{eq: decomposition in the urod case} with $\lambda$ replaced by $\lambda^*$ that $\W^\kk_f\otimes T^{\kk^o}_{0,\lambda^*}\subset M^*$ and thus $M^*= \cF_f(\W^\kk_f\otimes T^{\kk^o}_{0,\lambda^*})$. Hence, $M\simeq \cF_f(\W^\kk_f\otimes T^{\kk^o}_{0,\lambda})$, a contradiction. Therefore, $\cF_f(\W^\kk_f\otimes T^{\kk^o}_{0,\lambda})$ is a simple $\W^{\kk'}_f\otimes V_Q$-module.
Next, we show $\cF_f(\W^\kk_f\otimes T^{\kk^o}_{0,\lambda})$ is isomorphic to $\weyl_{\lambda, f}^{\kappa'} \otimes  V_{\lambda^*}$ as $\W^{\kk'}_f\otimes V_Q$-modules. 
Indeed, there is a non-zero homomorphism thanks to Corollary \ref{cor:urod} and Frobenius reciprocity:
\begin{align*}
    \Hom_{\W^{\kk'}_f\otimes V_Q}&(\cF_f(\W^\kk_f\otimes T^{\kk^o}_{0,\lambda}),\weyl_{\lambda, f}^{\kappa'} \otimes  V_{\lambda^*})\\
    &\simeq \bigoplus_{\mu \in P_+\cap Q} \Hom_{\W^{\kk}_f\otimes \W^{\kk^o}}(\W^\kk_f\otimes T^{\kk^o}_{0,\lambda}, \weyl_{\mu, f}^{\kappa} \otimes  T_{\mu,\lambda}^{\kk^o}) \simeq \C \Id.
\end{align*}
Since $\cF_f(\W^\kk_f\otimes T^{\kk^o}_{0,\lambda})$ is a simple $\W^{\kk'}_f\otimes V_Q$-module and is isomorphic to $\weyl_{\lambda, f}^{\kappa'} \otimes  V_{\lambda^*}$ as $\W^{\kk}_f\otimes \W^{\kk^o}$-modules, the non-zero homomorphism is an isomorphism. 
Thus, $\weyl_{\lambda, f}^{\kappa'}$ is a simple $\W^\kk_f$-module since $\weyl_{\lambda, f}^{\kappa'} \otimes  V_{\lambda^*}\simeq \cF_f(\W^{\kk'}_f\otimes T^{\kk^o}_{0,\lambda})$ is simple.

\noindent
(2) The assertion follows from the isomorphisms $\weyl_{\lambda, f}^{\kappa'} \otimes  V_{\lambda^*}\simeq \cF_f(\W^\kk_f\otimes T^{\kk^o}_{0,\lambda})$ and the fact that the induction functor $\cF_f$ is monoidal:
\begin{align*}
    (\weyl_{\lambda,f}^{\kk'}\otimes V_{\lambda^*}) \boxtimes (\weyl_{\mu,f}^{\kk'} \otimes V_{\mu^*})
    &\simeq \cF_f(\W^{\kk'}_f\otimes T^{\kk^o}_{0,\lambda}) \boxtimes  \cF_f(\W^{\kk'}_f\otimes T^{\kk^o}_{0,\mu})\\
    &\simeq \cF_f\left((\W^{\kk'}_f\otimes T^{\kk^o}_{0,\lambda}) \boxtimes  (\W^{\kk'}_f\otimes T^{\kk^o}_{0,\mu})\right)\\
    &\simeq \bigoplus_{\nu \in P_+}c_{\lambda,\mu}^\nu \cF_f(\W^{\kk'}_f\otimes T^{\kk^o}_{0,\nu})\\
    &\simeq \bigoplus_{\nu \in P_+}c_{\lambda,\mu}^\nu \weyl_{\nu,f}^{\kk'}\otimes V_{\nu^*}.
\end{align*}
In the above, we have used Theorem \ref{thm:fusion, regular case}. This completes the proof.
\end{proof}

\subsection{Equivalence of categories}
Let us denote the full subcategory $\KL_f^\kappa(\g)$ of $\W^\kk(\g,f)\Mod_{C_1}$ consisting of direct sums of the simple $\W^\kk(\g,f)$-modules $\weyl_{\lambda,f}^\kk$ ($\lambda \in P_+$).
By Theorem \ref{thm:fusion}, it is a (semisimple) rigid braided tensor category, whose branching rules of tensor products are the same as the Kazhdan-Lusztig category $\KL^\kk(\g)$ by the correspondence
$$H_f^0\colon \KL^\kk(\g) \xrightarrow{\simeq} \KL_f^\kappa(\g),\quad \weyl_\lambda^\kk \mapsto \weyl_{\lambda,f}^\kk.$$

 \begin{theorem}\label{thm:main}
 Let $\g$ be of type $ADE$ and $\kappa \in \C \backslash \Q$. Then, the BRST reduction $H_f^0\colon \KL^\kk(\g) \xrightarrow{\simeq} \KL_f^\kappa(\g)$ is an equivalence of braided tensor categories. In particular, $\KL_f^\kappa(\g)$ is rigid with  the dual of $\weyl^\kk_{\lambda,f}$ being $\weyl^\kk_{\lambda^*,f}$.
\end{theorem}
\begin{proof}
By Proposition \ref{prop: BRST reduction for intertwining operators} and Theorem \ref{thm:fusion}, we have the functorial isomorphisms 
$$H_f^0\colon I_{V^\kk(\g)}\binom{\weyl_{\nu}^\kk}{\weyl_{\lambda}^\kk\ \weyl_{\mu}^\kk}\xrightarrow{\simeq} I_{\W^\kk(\g,f)}\binom{\weyl_{\nu,f}^\kk}{\weyl_{\lambda,f}^\kk\ \weyl_{\mu,f}^\kk},$$
which give isomorphisms
\begin{align}\label{eq: tensor isom}
   H_f^0(\weyl_{\lambda}^\kk)\boxtimes H_f^0(\weyl_{\mu}^\kk)  \xrightarrow{\simeq} H_f^0(\weyl_{\lambda}^\kk\boxtimes \weyl_{\mu}^\kk).
\end{align}
Note that the definition $\W^\kk(\g,f)=H_f^0(V^\kk(\g))$ provides the identification of the unit objects. 
It is straightforward to check that the functorial isomorphisms \eqref{eq: tensor isom} satisfy the compatibility with the associativity, the left/ right unit identities, and the braiding isomorphism, which are identical at the level of intertwining operators, and thus vertex tensor categories. Hence, $H_f^0$ is a braided tensor equivalence. 

Indeed, in order to show the existence of braided tensor equivalence, we only have to apply the mirror equivalence in Theorem \ref{thm:main_thm} for the equivariant $\W$-algebras 
$$\cdo{G}{\kappa},\qquad \cdo{G,f}{\kappa}$$
as conformal extensions of the tensor products of pairs of simple vertex algebras 
$$(U,V)=(V^{\kk}(\g), V^{\kk_o}(\g)), (\W^{\kk}(\g,f), V^{\kk_o}(\g))$$
where the level $\kk_o \in \C \backslash \Q$ is given by the formula \eqref{eq: opposit level}.
Note that the condition (M1) follows from Theorem \ref{thm: equiv Walg} and (M2)-(M4) from Example \ref{ex: affine Kazhdan-Lusztig} and Theorem \ref{thm:fusion} by setting 
$\sC_U=\KL^\kk(\g), \KL^\kk_f(\g)$ and $\sC_V=\KL^{\kk_o}(\g)$. 
Moreover, the condition in Theorem \ref{thm:main_thm} on intertwining operators is also satisfied by Theorem \ref{thm:fusion}. Hence, Theorem \ref{thm:main_thm} implies that there are braided-reverse tensor equivalences 
\begin{align*}
    \KL^\kk(\g) \overset{\simeq}{\longleftarrow}\KL^{\kk_o}(\g) \overset{\simeq}{\longrightarrow}\KL^\kk_f(\g),
\end{align*}
and thus the assertion of the existence of a braided tensor equivalence.
\end{proof}
\begin{corollary}
    For $\g$ of type $ADE$, and $\kk\in \C \backslash \Q$, the braided tensor category of $\W^\kk(\g)$-modules consisting of the direct sums of $T^\kk_{\lambda,\mu}$ $(\lambda,\mu \in P_+)$ is rigid with the dual of $T^\kk_{\lambda,\mu}$ being $T^\kk_{\lambda^*,\mu^*}$. 
\end{corollary}
\begin{proof}
By Theorem \ref{thm:main}, $T^\kk_{\lambda,0}$ is rigid with $T^\kk_{\lambda^*,0}$ its dual, and so is $T^\kk_{0,\mu}$ with $T^\kk_{0,\mu^*}$ by Theorem \ref{thm: simplicity of T-modules}.
Then, since $T^\kk_{\lambda,\mu}\simeq T^\kk_{\lambda,0}\boxtimes T^\kk_{0,\mu}$ by Theorem \ref{thm:fusion, regular case}, it is rigid with $T^\kk_{\lambda^*,\mu^*}$ its dual, see e.g. \cite[Chapter 2.10]{EGNO}.
\end{proof}
\subsection{Comparison at different levels}
Thanks to Theorem \ref{thm:main}, the induction functor $\cF_f$ in \eqref{eq: induction functor 2} can be used to compare the Kazhdan-Lusztig categories $\KL^\kk_f(\g)$ explicitly.
We note that by the braided tensor equivalence $\KL^\kk_f(\g)\simeq U_q(\g)\mod$ in Example \ref{ex: affine Kazhdan-Lusztig}, we may compare them categorically in principle. However, there is a subtlety on the braided tensor structure on $U_q(\g)\mod$, that is, the choice of an $R$-matrix used for the braiding isomorphisms $L_q(\lambda)\otimes L_q(\mu)\simeq L_q(\mu)\otimes L_q(\lambda)$. This is not determined by the value $q=\exp{\left(\frac{1}{(\kk+h^\vee)}\right)}$, but rather the choice of a $|P/Q|$-root of $q$. The effect of different choices is rephrased as coupling with the category $V_Q\mod$ appearing in the functor $\cF_f$.  

By composing the duality in Theorem \ref{thm: simplicity of T-modules} and the functor $\cF_f$ in \eqref{eq: induction functor 2}, we obtain a braided tensor functor 
\begin{align*}
\Psi\colon \KL^{\check{\kk}^o}_{f_{\pr}}(\g) \hookrightarrow \W^{\kk^o}(\g)\Mod_{C_1}^{0} \overset{A_f \boxtimes (\W_f^\kk \otimes -)}{\longrightarrow} \KL^{\kk'}_f(\g)\boxtimes V_Q\mod
\end{align*}
where $f_{\pr}$ is a principal nilpotent element and $f$ is a general one.
By rewriting the levels $\check{\kk}^o$ and $\kk^o$ by $\kk$ and $\kk_1$, the latter is determined by the relation
\begin{align*}
    \frac{1}{\kk+h^\vee}=\frac{1}{\kk_1+h^\vee}+1.
\end{align*}
Recall that $V_Q\mod$ has simple objects $V_{\lambda}$ with $\lambda \in P/Q$ and is indeed braided tensor equivalent to the category $\vect_{P/Q}$ of the $P/Q$-graded spaces with braiding determined by the quadratic form $\mathrm{q}\colon P/Q\rightarrow \C^\times$ such that $\mathrm{q}(\lambda)\mapsto e^{2\pi i (\lambda,\lambda)}$. 
On the other hand, $\KL^{\kk'}_f(\g)\simeq \KL^{\kk'}(\g)$ has a canonical $P/Q$-grading.
It follows from Proposition \ref{prop: identification of induction} and Corollary \ref{cor:urod} that $\Psi(T^\kk_{0,\lambda})\simeq \weyl^{\kk_1}_{\lambda,f}\otimes V_{\lambda^*}$ and thus $\Psi$ induces an equivalence onto the trivial $P/Q$-grading part:
\begin{proposition}\label{prop: first comparison}
    For $\g$ of type $ADE$ and $\kk \in \C \backslash \Q$, $\Psi$ gives an braided tensor equivalence
    \begin{align*}
        \Psi\colon \KL^{\kk}_{f_{\pr}}(\g) \xrightarrow{\simeq} \left(\KL^{\kk_1}_f(\g) \boxtimes \vect_{P/Q}\right)^{P/Q}.
    \end{align*}
\end{proposition}
More generally, let $\vect_{P/Q}^m$ denote the braided tensor category of the $P/Q$-graded vector spaces with braiding determined by the quadratic form $\mathrm{q}^m(\lambda)\mapsto e^{2m\pi i (\lambda,\lambda)}$ for $m\in \Z$.
Note that it is realized as a subcategory of $V_{mQ}\mod$ consisting of $V_{m\lambda}$  ($\lambda \in P/Q$) when $m>0$. Then, we have the following.

\begin{corollary}\label{cor:main} 
For $\g$ of type $ADE$, $\kk \in \C \backslash \Q$, set $\kk_m \in  \C \backslash \Q$ by $$\frac{1}{\kk + h^\vee} = \frac{1}{\kk_m + h^\vee} +m,\quad (m\in \Z).$$
Then, there exists a braided tensor equivalence 
$$\KL^{\kk}_f(\g) \simeq \left(\KL_{f'}^{\kk_m}(\g)\boxtimes \vect^m_{P/Q}\right)^{P/Q}$$
for arbitrary nilpotent elements $f$ and $f'$.
\end{corollary}
We note that the case $f=f'=0$ essentially appears in a work of Moriwaki \cite{Mo} through the shifts of chiral differential operators by applying the mirror equivalence Theorem \ref{thm:main_thm}. Thus, the general case also follows from it by applying Theorem \ref{thm:main}. We give here an alternative proof based on Proposition \ref{prop: first comparison}.
\begin{proof}
By Theorem \ref{thm:main}, it suffices to show the case when $f=f'=f_{\pr}$.
Then Theorem \ref{prop: first comparison} gives the assertion for $m=1$.
Note that there are braided tensor equivalences 
\begin{align}\label{eq: diag embedding}
    \vect_{P/Q}^{m+n} \hookrightarrow (\vect_{P/Q}^m \boxtimes \vect_{P/Q}^n),\quad \C_\lambda \mapsto \C_{\lambda}\otimes \C_\lambda 
\end{align}
where $\C_{\lambda}$ is the one-dimensional vector space with grading $\lambda\in P/Q$. 
Now, by iterating Theorem \ref{prop: first comparison} on the right-hand side of the equivalence, we obtain
\begin{align*}
    \KL^{\kk}_{f_{\pr}}(\g) 
    &\simeq \left(\KL_{f_{\pr}}^{\kk_{m}}(\g) \vect^1_{P/Q})\boxtimes \vect^m_{P/Q}\right)^{P/Q}\\
    &\simeq \left((\KL_{f_{\pr}}^{\kk_{m+1}}(\g)\boxtimes \vect_{P/Q}^1)^{P/Q}\boxtimes \vect^m_{P/Q}\right)^{P/Q}\\
    &\simeq \left((\KL_{f_{\pr}}^{\kk_{m+1}}(\g)\boxtimes \vect^{m+1}_{P/Q}\right)^{P/Q},
\end{align*}
which proves the assertion for $m>0$.
The assertion $m=0$ holds trivially since $\vect^{0}_{P/Q}$ is the category of $P/Q$-graded vector spaces with trivial braiding. 
The case $m<0$ can be derived similarly from the case $m>0$:
\begin{equation*}
    \begin{split}
\left(\KL_{f_{\pr}}^{\kk}(\g) \boxtimes \vect^{-m}_{P/Q}\right)^{P/Q} &\simeq \left(\left(\KL_{f_{\pr}}^{\kk_m}(\g)\boxtimes \vect^{m}_{P/Q}\right)^{P/Q} \boxtimes \vect^{-m}_{P/Q}\right)^{P/Q} \\
&\simeq \left(\KL_{f_{\pr}}^{\kk_m}(\g)\boxtimes \vect^{0}_{P/Q} \right)^{P/Q} \\
&\simeq \KL_{f_{\pr}}^{\kappa_m}(\g)
 \end{split}
\end{equation*}
by using \eqref{eq: diag embedding}.
\end{proof}

\end{document}